\documentclass[12pt,reqno]{amsart}
\usepackage{fullpage}
\usepackage[top=1.5in, bottom=2in, left=2.5in, right=2.5in]{geometry}
\usepackage{graphicx}
\usepackage[draft]{hyperref}
\usepackage{amsmath,amsopn,amssymb,amsfonts,stmaryrd}
\usepackage{verbatim}
\usepackage{amsthm}
\usepackage{mathtools}
\usepackage{color}
\usepackage{enumitem}
\usepackage[framemethod=TikZ]{mdframed}
\usepackage{bbm}
\usepackage{mathrsfs}
\usepackage{booktabs}
\usepackage{caption}
\usepackage{bm}
\usepackage{tensor}
\usepackage{cleveref}

\newcommand{\Q}{\mathbb{Q}}

\renewcommand{\H}{\mathbb{H}}

\newcommand{\CO}{\mathcal{O}}

\renewcommand{\Im}{{\rm Im}}

\newcommand{\SL}{\mathrm{SL}}

\newcommand{\zz}{\mathfrak{z}}

\newcommand{\N}{\mathbb N}
\newcommand{\C}{\mathbb C}

\theoremstyle{plain}
\newtheorem{thm}{Theorem}[section]
\newtheorem*{theorem*}{Theorem}
\newtheorem{cor}[thm]{Corollary}
\newtheorem{lem}[thm]{Lemma}
\newtheorem{prop}[thm]{Proposition}

\newtheorem*{rem}{Remark}
\newtheorem*{rems}{Remarks}

\theoremstyle{definition}

\numberwithin{equation}{section}

\newcommand{\dg}{\ast}

\newcommand{\Z}{\mathbb Z}
\newcommand{\F}{\mathbb{F}}

\newif\ifdefs
\defstrue

\ifdefs
\else
\excludecomment{discussion}
\excludecomment{extradetails}
\fi

\setlist[itemize]{noitemsep, topsep=0pt}

\allowdisplaybreaks

\makeatletter
\newcommand{\vast}{\bBigg@{2}}
\newcommand{\Vast}{\bBigg@{5}}
\@namedef{subjclassname@2020}{\textup{2020} Mathematics Subject Classification}
\makeatother


\renewcommand{\pmod}[1]{\ \left( \mathrm{mod} \, #1 \right)}
\newcommand{\Pmod}[1]{\ ( \mathrm{mod} \, #1 )}

\newcommand{\ord}{\operatorname{ord}}

\author{Kathrin Bringmann}
\address{University of Cologne, Department of Mathematics and Computer Science, Weyertal 86-90, 50931 Cologne, Germany}
\email{kbringma@math.uni-koeln.de}
\author{Pavel Guerzhoy}
\address{Department of Mathematics, University of Hawaii, 2565 McCarthy Mall, Honolulu, HI, 96822-2273}
\email{pavel@math.hawaii.edu}
\author{Ben Kane}
\address{The University of Hong Kong, Department of Mathematics, Pokfulam, Hong Kong}
\email{bkane@hku.hk}
\author{Michael Mertens}
\address{Lehrstuhl f\"ur Algebra und Zahlentheorie, RWTH Aachen University, Pontdriesch 14/16, D-52062 Aachen}
\email{michael.helmut.mertens@rwth-aachen.de}
\author{Larry Rolen}
\address{Department of Mathematics, Vanderbilt University, Nashville, TN 37240}
\email{larry.rolen@vanderbilt.edu}

\title{On $U_p$-congruences for meromorphic modular forms via supersingularity}

\begin{document}
\date{\today}
\begin{abstract}
In this paper, we investigate congruences for meromorphic modular forms $F$  which have a pole at a single point $\zz$ in the fundamental domain of $\mathrm{SL}_2(\mathbb Z)$. For a prime $p$ with good supersingular reduction at the elliptic curve corresponding to $\zz$, we show that there exists a cusp form $f$ such that $F|U_p^m \equiv f|U_p^m \pmod{p^{\kappa_m}}$, where $\kappa_m=\alpha m -\beta$ with $\alpha$ only depending on the weight of $F$ and $\beta$ depending on $F$ and $p$ but is independent of $m$. In particular, if the space of cusp forms is trivial, then $F|U_p^m\equiv 0 \pmod{p^{\kappa_m}}$ vanishes $p$-adically to a high order. 
In order to prove these results, we use the fact that $p$ has supersingular reduction to realize $F$ as an overconvergent modular form and then utilize the theory of overconvergent forms to show the congruences.
\end{abstract}
\maketitle

\section{Introduction and statement of results}
Let $F$ be a meromorphic modular form of weight $k\in 2\N$ for $\SL_2(\Z)$ and write its Fourier expansion as\footnote{Throughout, Fourier expansions hold for $\Im(\tau)$ sufficiently large.} $F(\tau)=\sum_{n\gg -\infty} c_{F}(n) q^n$, where $q=e^{2\pi i\tau}$. Throughout, $p>3$ is prime. Define the {\it $U_p$-operator} by $F|U_p(\tau):=\sum_{n\gg -\infty} c_{F}(pn) q^n$. Congruences modulo powers of $p$ for repeated iterations of $U_p$ on $F$ have a long history for  holomorphic modular forms and for meromorphic forms with all poles at the cusps $\Q\cup\{i\infty\}$. For example, standard techniques using repeated differentiation of $\frac{E_8}{\Delta}$ yield
\begin{equation}\label{eqn:CongInfinity}
\left(\frac{E_6^3}{\Delta} +1488 E_6\right)\Big|U_{p}^m\equiv 0\pmod{p^{5m}},
\end{equation}
where $E_k(\tau):= 1- \frac{2k}{B_k} \sum_{n\ge1}\sum_{d\mid n} d^{k-1}q^n$ denotes the weight {\it $k$ Eisenstein series} and $\Delta(\tau) := \frac{E_4^3(\tau)- E_6^2(\tau)}{1728}$ is the \emph{discriminant modular form}. In this paper, we are interested in the case where  $F$ is a {\it meromorphic cusp form} with a pole at a single point $\zz$ in $\SL_2(\Z)\backslash\H$, i.e., that it is meromorphic with a pole at $\zz$ and a zero at $i\infty$.  For example, numerical calculations\footnote{Numerical calculations for this paper were done using PARI/GP  \cite{PARI2}.} suggest that the striking congruences 
\begin{align}\label{eqn:Delta/E6cong}
	\frac{\Delta}{E_6} \Big| U_p^{2m} \equiv 0 \pmod{p^{4m}},\\
\label{eqn:wt12cong}
\left(\frac{E_{4}^3\Delta}{E_6^2}-c_p\Delta\right)\Big|U_{p}^{2m}\equiv 0\pmod{p^{10m}}
\end{align}
hold for $p\equiv 3\pmod{4}$ and $m\in\N$, for some $c_p \in \Z_p$ depending on $p$. For instance, for $p=7$, calculations indicate that
\begin{multline*}
c_7\equiv 2 + 4 \cdot7^2 + 4\cdot7^4 + 5\cdot7^5 + 4\cdot7^6 + 7^7 + 6\cdot7^8 + 3\cdot7^9 + 2\cdot7^{10} + 3\cdot7^{11} + 5\cdot7^{13} + 6\cdot7^{14} \\+ 6\cdot7^{15} 
+ 6\cdot7^{16} + 6\cdot7^{19}  + 7^{20} + 4\cdot 7^{21} + 7^{22} + 4\cdot 7^{23} + 2\cdot 7^{24}\\
 + 3\cdot 7^{25} + 4\cdot 7^{26} + 6\cdot 7^{27} + 7^{28} + 6\cdot 7^{29} \pmod{7^{30}}
\end{multline*}
satisfies \eqref{eqn:wt12cong} for $m \leq 2$.

However, such congruences do not seem to occur for primes $p\equiv 1\pmod{4}$. To explain this discrepancy, we associate a point $\zz\in\SL_2(\Z)\backslash\H$ with the elliptic curve $E$ over $\C$ satisfying $j(E)=j(\zz)$, where $j:=\frac{E_4^3}{\Delta}$ is the $j$-invariant on $\H$ (we refer to \cite{Silverman,Silverman_a} for basic facts on elliptic curves). Observe that the weight six meromorphic cusp form $F_1=\frac{\Delta}{E_6}$ (as well as $F_2=\frac{E_4^3\Delta}{E_6^2} - c_p \Delta$ of weight $12$) has its only pole in the fundamental domain  $\SL_2(\Z)\backslash\H$ at the point corresponding to the elliptic curve with complex multiplication by $\Z[i]$. This curve has good supersingular reduction at every prime $p \equiv 3 \pmod 4$ and good ordinary reduction for the primes $p \equiv 1 \pmod 4$.

We next give another example of a form which numerical experiments suggest satisfies congruences similar to \eqref{eqn:Delta/E6cong} and \eqref{eqn:wt12cong}:
\begin{align} 
\label{eq:ex14_1}
\frac{\Delta E_4}{21952E_4^3- 9938375\Delta} \Big| U_5^{2m} &\equiv 0 \pmod{5^{2m}},\\
\frac{\Delta E_6}{21952E_4^3- 9938375\Delta} \Big| U_5^{2m}&\equiv 0 \pmod{5^{4m}},\nonumber
\end{align}
\begin{align} 
\frac{\Delta E_4^2}{21952E_4^3- 9938375\Delta} \Big| U_5^{2m} &\equiv 0 \pmod{5^{6m}},  \nonumber\\
\frac{\Delta E_4 E_6}{21952E_4^3- 9938375\Delta} \Big| U_5^{2m}&\equiv 0 \pmod{5^{8m}},\nonumber
\end{align}
\begin{equation} \label{eq:ex14_2}
\frac{\Delta E_4^2E_6}{21952E_4^3- 9938375\Delta} \Big| U_5^{2m} \equiv 0 \pmod{5^{12m}}.
\end{equation}
 Note that the meromorphic modular forms to which we apply the $U_p$-operator have their poles at the point corresponding to the elliptic curve of conductor $14$ (this curve is usually labeled as \verb|14C|,  with $j$-invariant $\frac{9938375}{21952}$, see e.g. \cite{Birch_Kuyk}); this elliptic curve is defined over $\Q$ and has good supersingular reduction at the prime $p=5$. 

 For a ring $R$, $\Gamma\subseteq\SL_2(\Z)$, and $k\in 2\N$, we let $S_{k}(\Gamma,R)$ denote the $R$-module of cusp forms with Fourier coefficients in $R$, omitting the dependence on $R$ if $R=\C$. We note that $S_{k}(\SL_2(\Z),\Z_p)=S_{k}(\SL_2(\Z),\Z)\otimes\Z_p$ is spanned over $\Z_p$ by a finite number of forms with integer coefficients because $S_k:=S_{k}(\SL_2(\Z),\C)$ has a basis of forms in $S_{k}(\SL_2(\Z),\Z)$ (see \cite[Theorem X.4.4]{Lang_mf}).

As noted above, the proof of \eqref{eqn:CongInfinity} takes advantage of the fact that $\frac{E_6^3}{\Delta} +1488 E_6$ is the derivative of a negative weight form. However, since differentiation increases the order of the pole at a point in $\H$, one sees immediately that $\frac{\Delta}{E_6}$ cannot be the derivative of a meromorphic modular form, and one can similarly show that the functions appearing in \eqref{eqn:Delta/E6cong}, \eqref{eqn:wt12cong}, and \eqref{eq:ex14_1}--\eqref{eq:ex14_2} also cannot be derivatives. Moreover, such results for derivatives are uniform for all primes and the observations here depend on the fact that $p$ is supersingular. These differences hint at the fact that a different approach is required for congruences involving meromorphic modular forms. Our main result is the following theorem which explains the numerical observations above.
\begin{thm} \label{thm_average_rate}
Suppose that the elliptic curve corresponding to $\zz\in \SL_2(\Z)\backslash\H$ is defined over $\Q$ and has good supersingular reduction at the prime $p\geq 5$. Let $F$ be a meromorphic cusp form of weight $k$ with Fourier coefficients in $\Z$ and whose only pole in $\SL_2(\Z)\backslash\H$ occurs at $\zz$. Then there exist $f \in S_{k}(\SL_2(\Z),\Q_p)$ and $\varepsilon\geq 0$ (depending on $F$) such that, for every $m\in\N_0$, 
\[
(F-f)|U_p^m \equiv 0 \pmod {p^{\frac{(k-2)m}2-\varepsilon}}.
\]
Moreover, for each $m$ there exists $f_m\in S_k(\SL_2(\Z),\Q)$ satisfying 
\[
f_m\equiv f\pmod{p^{\frac{(k-2)m}{2}}}\text{, } \quad
\left(F-f_m\right)|U_p^m \equiv 0 \pmod {p^{\frac{(k-2)m}2-\varepsilon}}.
\]
\end{thm}
\begin{rems}
\noindent

\noindent
\begin{enumerate}[leftmargin=*]
\item[\rm(1)]
Computations indicate that similar congruences hold for $p\in\{2,3\}$, but some technical modifications would be necessary to deal with these primes. 
\item[\rm(2)]
Congruences that only hold up to an $\varepsilon$-power, such as in Theorem \ref{thm_average_rate}, appear in multiple places. For instance, see \cite[Example 4.3.1]{Calegari} (our $\varepsilon$ is called $c$ there).
\item[\rm (3)]
Within the range of our numerical calculations, it appears that $\varepsilon$ is always $0$. However,  our methods do not allow us to show $\varepsilon=0$ in any specific case (see the remark below Proposition \ref{thm:TruncatedAsymptotic}). The congruences from Theorem  \ref{thm_average_rate} with $\varepsilon=0$ were proven by Paşol and Zudilin \cite{PS} for the meromorphic modular forms $\frac{\Delta}{E_4^2}$, $\frac{\Delta E_4}{E_6^2}$, and $\frac{\Delta E_6}{E_4^3}$. More recently, a family of meromorphic modular forms was investigated by Zhang \cite{Zhang} using the methods introduced in \cite{PS}. Zhang proved Theorem \ref{thm_average_rate} with $\varepsilon=0$ (in all these cases, $f=0$) for these forms in \cite[Theorems 4.7 and 6.1]{Zhang}. The methods introduced in \cite{PS} and developed in \cite{Zhang} are different from ours, and apply only to these very special meromorphic modular forms. In particular, for these forms, the congruences hold for all primes, while numerical examples show that the conditions on $p$ in Theorem \ref{thm_average_rate} are necessary in general.
\end{enumerate}
\end{rems}
As evidenced in \eqref{eqn:Delta/E6cong} and \eqref{eq:ex14_1}--\eqref{eq:ex14_2}, the congruences from Theorem \ref{thm_average_rate} are particularly nice if $S_k=\{0\}$. In this case $f=0$ in Theorem \ref{thm_average_rate}, which immediately implies the following corollary. 
\begin{cor} \label{cor_main}
Suppose that $k\in \{2,4,6,8,10,14\}$ and that the elliptic curve corresponding to $\zz\in \SL_2(\Z)\backslash\H$ has good supersingular reduction at the prime $p$. Then, for a meromorphic cusp form $F$ of weight $k$ with integral Fourier coefficients whose only pole in $\SL_2(\Z)\backslash\H$ occurs at $\zz$, there exists $\varepsilon\geq 0$ such that, for all $m\in\N_0$, 
\[
F|U_p^m  \equiv 0 \pmod {p^{\frac{(k-2)m}{2}-\varepsilon}}.
\]
\end{cor}
\begin{rem}
Corollary \ref{cor_main} implies that, for $p \equiv 3 \pmod 4$, we have 
\begin{equation*}
	\frac\Delta{E_6} \big| U_p^{2m} \equiv 0 \pmod{p^{4m-\varepsilon}}.
\end{equation*}
Up to the $\varepsilon$ in the exponent, this agrees with the congruences experimentally observed in \eqref{eqn:Delta/E6cong}. Similarly, Corollary \ref{cor_main} explains \eqref{eq:ex14_1} and \eqref{eq:ex14_2}.
\end{rem}

The paper is organized as follows.  In Section \ref{sec:prelims}, we recognize the meromorphic modular form $F$ from Theorem \ref{thm_average_rate} as an overconvergent modular form and recall basic facts from the theory of overconvergent modular forms. In Section \ref{sec:thmaverage}, we prove Theorem \ref{thm_average_rate} and Corollary \ref{cor_main}.

\section*{Acknowledgements}
The first author has received funding from the European Research Council (ERC) under the European Union’s Horizon 2020 research and innovation programme (grant agreement No. 101001179). The research of the third author was supported by grants from the Research Grants Council of the Hong Kong SAR, China (project numbers HKU 17314122 and HKU 17305923). The fifth author was supported by a grant from the Simons Foundation (853830, LR).

\section{Preliminaries}\label{sec:prelims}
\subsection{Meromorphic modular forms}
For $k\in 2\Z$, $\gamma=\left(\begin{smallmatrix}a&b\\c&d\end{smallmatrix}\right)\in\SL_2(\Z)$, and $F\colon\mathbb H\rightarrow\mathbb C$, we define the   \begin{it}weight $k$ slash operator\end{it} by
\[
F\big|_{k}\gamma(\tau):=(c\tau+d)^{-k}F(\gamma \tau).
\]
For $\Gamma\subseteq \SL_2(\Z)$, a meromorphic function $F\colon\H\to\C$ is a \begin{it}meromorphic modular form of weight $k$ on $\Gamma$\end{it} if $F|_{k}\gamma=F$ for all $\gamma\in\Gamma$ and, for each $\gamma\in\SL_2(\Z)$, the Fourier expansions of $F$ at the \begin{it}cusp\end{it} $\sigma:=\gamma^{-1}i\infty$ has the shape 
\[
F|_k\gamma(\tau)=\sum_{n\gg -\infty}c_{F,\sigma}(n) q^{\frac{n}{N_{\sigma}}},
\]
where $N_{\sigma}$ is the \begin{it}cusp width\end{it} of $\sigma$. 
If $F$ is also holomorphic on $\H$ and $c_{F,\sigma}(n)=0$ for all $n<0$ and cusps $\sigma$, then we call $F$ a \begin{it}holomorphic modular form (of weight $k$ on $\Gamma$)\end{it}. We denote the space of holomorphic modular forms of weight $k$ by $M_k(\Gamma)$. If $F\in M_k(\Gamma)$ and $c_{F,\sigma}(0)=0$ for all $\sigma$, then we call $F$ a \begin{it}cusp form (of weight $k$ on $\Gamma$)\end{it}. 

We also require the \emph{$V_p$-operator} acting on the Fourier expansion of a weight $k$ meromorphic modular form $F(\tau)=\sum_{n\gg -\infty} c_F(n)q^n$ by
\[
F|V_p(\tau):=\sum_{n\gg -\infty}c(n)q^{pn}.
\]
We also use the \emph{Hecke operators} 
\[
F|T_p(\tau):=\sum_{n\gg -\infty}\left(c(pn)+p^{k-1}c\left(\frac{n}{p}\right)\right)q^n=F|U_p(\tau)+p^{k-1}F|V_p(\tau).
\]
\subsection{Overconvergent modular forms}
Although Theorem \ref{thm_average_rate} only involves congruences between meromorphic modular forms and cusp forms, we require a larger space to investigate the congruences between them. Namely, under the conditions of Theorem \ref{thm_average_rate}, we embed them into the space of so-called overconvergent modular forms because this space has a nice spectral theory under $U_p$ that can be exploited to prove congruences. We first recall the definition of these objects.

We let $\C_p$ denote the $p$-adic completion of the algebraic closure of $\Q_p$ and $\CO_p := \{ z \in \C_p  :  \ord_p(z)\geq 0 \} \subset \C_p$ the {\it ring of integers} in $\C_p$. Recall that the space $M_k^{\dg}(\CO_p, r)$ of overconvergent forms of level $N$ is defined as the space of ``functions'' on certain ``test objects'' of level $N$, which in the case of trivial (see \cite[Definition 2]{Gouvea_cont}) auxiliary level $N=1$ are pairs $(E, \omega)$, where $E$ is an elliptic curve defined over $\CO_p$, and $\omega$ is the (unique up to a constant multiple) holomorphic differential on $E$. We refer the reader to \cite[Section 2]{Gouvea_cont} for the definitions and further references, and to the more recent survey \cite{Calegari} for an enlightening discussion supplied with an abundance of explicit examples. 

We next describe a growth rate related to overconvergent modular forms. For this, we require a lifting $g$ of the Hasse invariant to characteristic zero, which we next describe. Classically, the Hasse invariant is defined (see e.g. \cite[Section V.3]{Silverman}) for elliptic curves over finite fields, taking value one for ordinary elliptic curves and zero for supersingular elliptic curves. One extends this definition to elliptic curves over characteristic $p$ rings, and modifies it (see e.g. \cite[Section 1.6]{Calegari}) in a way such that the Hasse invariant becomes an algebraic modular form (in the sense of Katz, in characteristic $p$) of weight $p-1$ for ${\rm SL}_2(\Z)$. A seminal calculation of Deligne (see e.g. \cite[Theorem 1.7.2]{Calegari}) showed that the $q$-expansion of this modular form $A$ is simply $A(T(q),\omega_{\operatorname{can}})=1$, where $T(q)$ is the Tate curve defined in \cite[Section 1.2.26]{Calegari} and $\omega_{\operatorname{can}}$ is the canonical differential. Note that, by the $q$-expansion principle, this condition determines the Hasse invariant (as a Katz modular form in characteristic $p$) uniquely.  A (complex-analytic, therefore algebraic) modular form $g$ of weight $p-1$ for the full modular group ${\rm SL}_2(\Z)$ such that $g \equiv 1 \pmod {p}$ as $q$-expansions is called a \begin{it}lifting of the Hasse invariant\end{it} (to characteristic zero). Initially, for $p \geq 5$, the Eisenstein series $E_{p-1}$ was chosen as a lifting of the Hasse invariant (see e.g. \cite[Theorem 1.8.1]{Calegari}), but the theory remains unchanged if $E_{p-1}$ is replaced by an arbitrary lifting $g$.

 The domain of an overconvergent form of growth rate $r$ (being overconvergent is equivalent to $\ord_p(r)>0$) consists of those pairs $(E,\omega)$ for which (see \cite[Definition 2]{Gouvea_cont} there exists $Y \in \CO_p$ such that $YG(E,\omega) = r$.
This condition can be rewritten as 
\[
\ord_p(G(E,\omega)) \leq \ord_p(r).
\]
In other words, the domain of an overconvergent form of growth rate $r$ consists of all pairs $(E,\omega)$, with discs around the supersingular points (those with $E$ supersingular) cut out. We use the modular form $g$ of weight $p-1$ constructed in the following lemma as our choice of a lifting of the Hasse invariant.
\begin{lem}\label{lem:supersingular}
	Suppose that for $\zz\in\SL_2(\Z)\backslash \H$, the elliptic curve $E$ corresponding to $\zz$ is defined over $\Q$. Let $p\geq 5$ be a  prime such that $E$ has good supersingular reduction at $p$. Then there exists $g\in M_{p-1}$ with Fourier coefficients in $\overline{\Q}$ such that the Fourier expansion satisfies $g(\tau)\equiv 1\pmod{p}$ and $g(\zz) = 0$.
\end{lem}
\begin{proof}
Write $p-1=12m+4\delta+6\epsilon$ with  $m,\delta\in\N_0$, $\delta\in\{0,1,2\}$, and $\epsilon\in\{0,1\}$ (this is possible because $p\geq 5$). By the valence formula,
 \begin{equation}\label{eqn:modularfunction}
\frac{E_{p-1}}{\Delta^mE_4^\delta E_6^{\epsilon}}
\end{equation}
is a weakly holomorphic modular function whose only possible pole is at $i\infty$. The $j$-function $j:=\frac{E_4^3}{\Delta}$ is a Hauptmodul for $\SL_2(\Z)$, and by \cite[(2) and Theorem 1]{KZ} there exists a rational polynomial $\mathcal{E}_{p-1}$ with $p$-integral coefficients (it is in general a rational function for a meromorphic function, but it is a polynomial because the only possible pole of \eqref{eqn:modularfunction} is at $i\infty$) such that
\[
E_{p-1}=\Delta^mE_4^\delta E_6^\epsilon \mathcal{E}_{p-1}(j).
\]

By \cite[Theorem 1]{KZ} we have\footnote{This is written as a congruence between polynomials in \cite[Theorem 1]{KZ}, but this is the same because we have characteristic $p$ in $\F_p$.}
\rm
\begin{equation}\label{eqn:SSproduct}
\pm j^\delta (j-1728)^{\epsilon} \mathcal{E}_{p-1}(j)=\prod_{\substack{ E/\overline{\mathbb F}_p \\ E\textrm{ supersingular}}}(j-j(E))\in\F_p[j].
\end{equation}
Since the $j$-function induces a bijection between $\SL_2(\Z)\backslash\overline{\H}$ (where $\overline{\H}:=\H\cup\Q\cup\{i\infty\}$) and $\overline{\C}:=\C\cup\{\infty\}$, there exists precisely one $z\in \SL_2(\Z)\backslash\H$ with $j(z)=j(\zz)$, and by assumption the corresponding $E$ is supersingular. Moreover, as noted below \cite[(2)]{KZ}, the product on the right-hand side of \eqref{eqn:SSproduct} is divisible by $j^\delta (j-1728)^\epsilon$, so we may write
\[
\prod_{\substack{ E/\overline{\mathbb F}_p \\ E\textrm{ supersingular}}}(j-j(E))=j^\delta (j-1728)^\epsilon (j-j(\zz)) \prod_{\substack{ E/\overline{\mathbb F}_p \\ E\textrm{ supersingular}\\ j(E)\not\equiv 0,1728,j(\zz)\pmod{p}}}(j-j(E)).
\]
Hence
\begin{align*}
g:=&\ \Delta^mE_4^\delta E_6^\epsilon (j-j(\mathfrak z))\!\!\!\prod_{\substack{E/\overline{\mathbb F}_p \\ E\textrm{ supersingular} \\ j(E)\not\equiv 0,1728,j(\mathfrak z)\pmod{p}}} \!\!(j-j(E))
\equiv \Delta^mE_4^\delta E_6^\epsilon\mathcal{E}_{p-1}(j)
\equiv E_{p-1}\\
\equiv&\  1\pmod{p}.
\end{align*}
Note that $g$ has Fourier coefficients in $\overline{\Q}$ because $j(\zz)$ and $j(E)$ are all in $\overline{\Q}$, while the Fourier coefficients of $j$, $\Delta$, $E_4$, and $E_6$ are integers. Moreover, $g(\zz)=0$ because of the factor $ j-j(\zz)$.\qedhere
\end{proof}
Throughout the paper, we replace $E_{p-1}$ in the standard definitions related to overconvergent modular forms with the lifting $g$ of the Hasse invariant from Lemma \ref{lem:supersingular}. We let $M^{\dg}_k(r):=M_k^{\dg}(\CO_p, r)$ be the space of $r$-overconvergent modular forms of weight $k$ and trivial level. Following \cite[p. 89]{Gouvea_cont}, observe that these are the formal Katz expansions
\begin{equation}\label{eqn:KatzExpansion}
M_k^{\dg}(r) = \left\{ \sum_{m\ge0} r^m b_m(E,\omega) g^{-m}(E,\omega) \right\}
\end{equation}
with certain holomorphic modular forms $b_m$ (defined over $\CO_p$ as algebraic modular forms) satisfying $b_m \rightarrow 0$ as $m \rightarrow \infty$. 
While this condition on $b_m$ already guarantees the convergence of the Katz expansion \eqref{eqn:KatzExpansion}, the presence of $r$ with $\ord_p(r)>0$ makes the convergence even stronger, which is the meaning of the term overconvergent.

We next recall some properties of $M_k^{\dg}( r)$.
Assume  that\footnote{The space $M_k^{\dg}(r)$ only depends on the $p$-adic order of $r$, not on $r$ itself (see \cite[p. 90]{Gouvea_cont}). This leads to a discrepancy in notation in the literature where sometimes $r$ is written for $\ord_p(r) \in \Q$ (see e.g. \cite{Calegari}).}
\begin{equation} \label{eq_growth_inequ}
\frac{1}{p+1} < \ord_p(r) < \frac{p}{p+1}.
\end{equation}
Given \eqref{eq_growth_inequ}, the space $M_{k}^{\dg}(r)$ admits a spectral decomposition under $U_p$, with the kernel being trivial by the left inequality (see \cite[Lemma 3.8.5]{Calegari}) and the right inequality implies that $U_p$ is completely continuous (see e.g. \cite[Definition 3.5.1]{Calegari} for a discussion and \cite[Proposition 1]{Gouvea_cont} for the statement itself). General theory of $p$-adic Banach spaces then implies that  $U_p$ has a spectrum of non-zero (generalized) eigenvalues 
\[
\ord_p\left(\lambda_1\right) \leq \ord_p\left(\lambda_2\right) \leq \ldots
\]
with $\ord_p(\lambda_m)\to\infty$ as $m \rightarrow \infty$. The quantities $\ord_p(\lambda_m)$ are called {\sl slopes}, and $\mathcal{Y}=\{y_m\}=\{\ord_p(\lambda_m)\}$ is the {\it slope sequence} (these are indeed the slopes of the characteristic power series of $U_p$ acting on  $M_k^\dg(r)$). The sequence (and the eigenvalues themselves) is independent of $r$ if \eqref{eq_growth_inequ} is satisfied. The numbers $y_j$ are non-negative rational numbers (see \cite[Section 3]{Gouvea_Mazur}). 

For every slope $y \in \mathcal{Y}=\{y_j\}$, consider the subspace $B^{[y]} \subset M_k^\dg(r)$ of forms of slope $y$ (see \cite[p520]{Gouvea_Mazur} for a definition). Since $U_p$ is completely continuous, $B^{[y]}$ is finite-dimensional and there exists a continuous projection $e_y\colon M_k^\dg(r) \rightarrow B^{[y]}. \text{ For } H \in M_k^\dg(r)$, the (formal as it may not converge) infinite series $\sum_{y \in \mathcal{Y}}  e_y(H)$ is called the {\it asymptotic expansion} of $H$. By \cite[Proposition 1]{Gouvea_Mazur}, for $H\in M_k^{\dg}(r)$, the truncated asymptotic expansion
\begin{equation} \label{eq:tr_as}
e_{\leq M}(H):=\sum_{\substack{y\in\mathcal{Y}\\ y\leq M}} e_y(H)
\end{equation}
closely approximates $H$ under repeated actions of $U_p$. Letting $\nu_p$ denote the $p$-adic norm in the Banach space $M_k^{\dg}(r)$, \cite[Proposition 1]{Gouvea_Mazur} implies the following.
\begin{prop}\label{thm:TruncatedAsymptotic}
Let $F\in M_k^{\dg}(r)$, where $r$ satisfies \eqref{eq_growth_inequ} and $M\geq 0$. Then there exists a constant $\nu>0$ for which
\begin{equation*} 
\nu_p\left( \left(F- e_{\leq M}(F)\right)|U_p^m\right) \geq  m(M+\nu)\quad\text{for } m \gg 0.
\end{equation*}
In particular,
\[
F\big|U_p^m\equiv e_{\leq M}(F)\big|U_p^m\pmod{p^{m(M+\nu)}}\quad\text{for } m \gg 0.
\]
\end{prop}

\begin{rem}
The expansions we consider here are only asymptotic;  the nature of the method used does not allow for an improvement $m>0$ instead of merely $m \gg 0$. That is one reason why specific congruences coming from computer calculations are slightly stronger than those in Theorem \ref{thm_average_rate}. In other words, our methods indicate the rate of convergence correctly, while they fail to prove the $\varepsilon=0$ observed in computer calculations.
\end{rem}

In addition to bounding the rate at which $e_{\leq M}(F)|U_p^m$ converges to $F|U_p^m$ as a function of $m$, we next show that Proposition \ref{thm:TruncatedAsymptotic} gives us some control over the field containing $e_{\leq M}(F)$. 
\begin{prop}\label{prop:ProjectField}
Let $M\geq 0$ and let $K/\Q_p$ be  a finite extension. Suppose that there exist eigenfunctions $f_1,\dots,f_d$ under $U_p$ which span $B^{[y]}$ for every slope $y\leq M$ and which have coefficients and eigenvalues in a finite extension $K/\Q_p$. Then, for any $F\in M_k^\dg(r)$ whose coefficients also lie in $K$, the truncated asymptotic expansion $e_{\leq M}(F)$ has coefficients in $K$.
\end{prop}
\begin{proof} 
If $d=0$, then we have $e_{\leq M}(F)=0$, so the claim holds trivially. 

Next suppose that $d\in\N$ and let $f_1,\dots,f_d$ be eigenfunctions spanning $B^{[y]}$ for $y\leq M$. We split $F$ as
\begin{equation}\label{eqn:splitF}
F=e_{\leq M}(F)+ \left(F-e_{\leq M}(F)\right).
\end{equation}
Since $f_1,\dots,f_d$ span the subspace of slopes $y\leq M$ over $\C_p$ and $e_{\leq M}(F)$ is contained in this space, there exist $a_1(F),\dots,a_d(F)\in\C_p$ such that 
\begin{equation}\label{eqn:e<=Mgenerators}
e_{\leq M}(F)=\sum_{j=1}^d a_j(F) f_j.
\end{equation}
By assumption, each $f_j$ has coefficients in $K$, so it suffices to show that $a_j(F)\in K$. For this, we fix $1\leq j\leq d$ and show that $a_j(F)\in K$.  Let $\lambda_j$ be the eigenvalue of $f_j$. Define
\begin{equation}\label{eqn:Tjtdef}
T_{j}:=\prod_{\substack{1\leq \ell\leq d\\ \lambda_{\ell}\neq \lambda_j}}\left(U_p-\lambda_{\ell}\right),
\end{equation}
which is a scalar multiple of the projection into the eigenspace containing $f_j$ if restricted to the finite-dimensional space spanned by $f_1,\dots,f_d$. Since $T_{j}$ is a polynomial in $U_p$, it commutes with $U_p$ and it is also a linear operator. Using the fact that $T_{j}$ and $U_p$ are both linear and commute, \eqref{eqn:splitF} implies that 
\begin{align}
\nonumber F_{j,m}&:=F|T_{j}|U_p^m= e_{\leq M}(F)|T_{j}|U_{p}^m + \left(F-e_{\leq M}(F)\right)|T_{j}|U_{p}^m\\
&= e_{\leq M}(F)|T_{j}|U_{p}^m + \left(F-e_{\leq M}(F)\right)|U_{p}^m|T_{j}.\label{eqn:F|TjUpm}
\end{align}

First note that $F_{j,m}$ has coefficients in $K$ because $F$ has its coefficients in $K$ and the operators $T_{j}$ and $U_{p}^m$ both preserve this property (as $\lambda_{\ell}\in K$). 
To evaluate the right-hand side of \eqref{eqn:F|TjUpm}, we expand the first and second terms in two different ways. Expanding $T_{j}=\sum_{\mu=0}^{s} b_{\mu} U_p^{\mu}$, the second term on the right-hand side of \eqref{eqn:F|TjUpm} can be written as 
\[
\left(F-e_{\leq M}(F)\right)|U_{p}^m|T_{j}=\sum_{\mu=0}^s b_{\mu}\left(F-e_{\leq M}(F)\right)|U_{p}^{m+\mu}.
\]
Since $b_{\mu}$ is a linear combination of products of $\lambda_{\ell}$ (homogeneous with degeree $s-\mu$) and $\ord_{p}(\lambda_{\ell})\geq 0$, we see that $\ord_{p}(b_{\mu})\geq 0$, so Proposition \ref{thm:TruncatedAsymptotic} implies that, for $m\gg 0$,  
\[
\nu_{p}\left(b_{\mu}\left(F-e_{\leq M}(F)\right)|U_{p}^{m+\mu}\right)\geq m(M+\mu+\nu).
\]
Thus, for $m\gg 0$, we conclude that  
\begin{align}\label{eqn:ordsecondTj}
\nu_p\left(\left(F-e_{\leq M}(F)\right)|U_{p}^m|T_{j}\right)&=\nu_p\left(\sum_{\mu=0}^s b_{\mu}\left(F-e_{\leq M}(F)\right)|U_{p}^{m+\mu}\right)\nonumber\\
&\geq m(M+\nu).
\end{align}
Plugging \eqref{eqn:ordsecondTj} into \eqref{eqn:F|TjUpm}, for $m\gg 0$, we conclude that 
\begin{equation}\label{eqn:Fjtmcong}
\nu_p\left(F_{j,m}- e_{\leq M}(F)|T_{j}|U_{p}^m\right)\geq m(M+\nu).
\end{equation}
We next simplify the second term in \eqref{eqn:Fjtmcong}. Plugging in \eqref{eqn:e<=Mgenerators}, we have
\[
e_{\leq M}(F)|T_{j}|U_{p}^m= \sum_{\ell=1}^d a_{\ell}(F) f_{\ell}|T_{j}|U_{p}^m.
\]
If $\lambda_{\ell}\neq \lambda_{j}$, then $f_{\ell}|T_{j}=0$, so 
\[
e_{\leq M}(F)|T_{j}|U_p^m=\sum_{\substack{1\leq \ell\leq d\\ \lambda_{\ell}=\lambda_j}}a_{\ell}(F)f_{\ell}|\prod_{\substack{1\leq L\leq d\\ \lambda_{L}\neq \lambda_{\ell}}}\left(U_p-\lambda_{L}\right)|U_p^m.
\]
On the other hand, if $\lambda_{\ell}=\lambda_j$, then $f_{\ell}|U_p^m=\lambda_j^m f_{\ell}$. Hence, setting $\alpha_j:=\prod_{\substack{1\leq L\leq d\\ \lambda_{L}\neq \lambda_{j}}} \,(\lambda_{j}-\lambda_{L})$, we conclude that 
\begin{equation}\label{eqn:firstTjm}
e_{\leq M}(F)|T_{j}|U_{p}^m=\alpha_j\lambda_{j}^{m}\sum_{\substack{1\leq \ell\leq d\\ \lambda_{\ell}=\lambda_j}}a_{\ell}(F)f_{\ell}.
\end{equation}
Plugging this into \eqref{eqn:Fjtmcong} yields 
\[
\nu_{p}\left(F_{j,m}- \alpha_j\lambda_{j}^{m}\sum_{\substack{1\leq \ell\leq d\\ \lambda_{\ell}=\lambda_j}}a_{\ell}(F)f_{\ell}\right)\geq  m(M+\nu).
\]
We then divide by $\alpha_j\lambda_j^m$ to conclude that, for $m\gg 0$, 
\[
\nu_p\left(\alpha_j^{-1}\lambda_j^{-m}F_{j,m}- \sum_{\substack{1\leq \ell\leq d\\ \lambda_{\ell}=\lambda_j}}a_{\ell}(F)f_{\ell}\right)\geq m\left(M+\nu-\ord_{p}\left(\lambda_j\right)\right)-\ord_p\left(\alpha_j\right).
\]
Since $\ord_{p}(\lambda_j)\leq M$, for $m\gg 0$ we conclude that
\[
\nu_p\left(\alpha_j^{-1}\lambda_j^{-m}F_{j,m}- \sum_{\substack{1\leq \ell\leq d\\ \lambda_{\ell}=\lambda_j}}a_{\ell}(F)f_{\ell}\right)\geq  m\nu-\ord_p\left(\alpha_j\right).
\]
Taking $m\to\infty$, we see that 
\[
\lim_{m\to\infty}\left(\alpha_j^{-1}\lambda_j^{-m}F_{j,m}\right)=\sum_{\substack{1\leq \ell\leq d\\ \lambda_{\ell}=\lambda_j}}a_{\ell}(F)f_{\ell},
\]
where the limit is taken in the $p$-adic norm in the Banach space. The coefficients in the $q$-expansion of $\alpha_j^{-1}\lambda_j^{-m}F_{j,t,m}$ must also converge to the coefficients in the $q$-expansion of $\sum_{\substack{1\leq \ell\leq d\\ \lambda_{\ell}=\lambda_j}}a_{\ell}(F)f_{\ell}$ (see e.g. \cite[Exercise 3.4.5]{Calegari}). We hence conclude that for any $n\in\N$ we have 
\begin{equation}\label{eqn:congalmost}
\lim_{m\to\infty} \left(\alpha_j^{-1}\lambda_j^{-m}c_{F_{j,m}}(n)\right)=\sum_{\substack{1\leq \ell\leq d\\ \lambda_{\ell}=\lambda_j}}a_{\ell}(F)c_{f_{\ell}}(n).
\end{equation}
Without loss of generality, we assume that $\lambda_{\ell}=\lambda_j$ for $1\leq \ell\leq d_j$, where $d_j$ is the number of such $\ell$. Since $f_1,\dots,f_{d_j}$ are linearly independent, we may choose $n_1,\dots,n_{d_j}$ to construct an invertible $d_j\times d_j$ matrix $A$ whose $\ell$-th column is given by $(c_{f_{\ell}}(n_1),c_{f_{\ell}}(n_2),\dots,c_{f_{\ell}}(n_{d_j}))^T$.
Defining the vector 
\begin{equation*}
\bm{C_m}:=\left(
\frac{c_{F_{j,m}}\left(n_1\right)}{\alpha_j\lambda_j^m},
\frac{c_{F_{j,m}}\left(n_2\right)}{\alpha_j\lambda_j^m},
\dots,
\frac{c_{F_{j,m}}\left(n_{d_j}\right)}{\alpha_j\lambda_j^m}
\right)^T.
\end{equation*}
and letting $\bm{a}$ be the vector whose $\ell$-th component is $a_{\ell}(F)$, \eqref{eqn:congalmost} implies that 
\[
\lim_{m\to\infty}\left(A^{-1} \bm{C_m}\right)=\bm{a}.
\]
Since $A$ and $\bm{C_m}$ both have coefficients in $K$, the components $a_{\ell,m}(F)$ of the vector $A^{-1}\bm{C_m}$ lie in $K$ as well. Thus the sequence $a_{\ell,m}(F)\in K$ converge to $a_{\ell}(F)$ in $\C_p$. However, $K$ is a finite extension of $\Q_p$, and it is well-known (for example, see \cite[Proposition 3 of Chapter II]{Serre}) that a finite extension of $\Q_p$ is complete, so a convergent sequence of elements of a finite extension must converge to an element of the same field. This implies that $a_{\ell}(F)\in K$. In particular, $a_j(F)\in K$. Since $j$ is arbitrary, we conclude that $a_j(F)\in K$ for all $1\leq j\leq d$, and hence 
\[
e_{\leq M}(F)=\sum_{j=1}^d a_j(F) f_{j}
\]
has coefficients in $K$ as well.\qedhere
\end{proof}

Taking linear combinations over $\C_p$ of the basis elements of $S_{k}(\SL_2(\Z),\Z)$ (see \cite[Theorem X.4.4]{Lang_mf}), we embed $S_k(\Gamma_0(p),\overline{\Q})$ into the space $S_{k}(\Gamma_0(p),\C_p)$ spanned over $\C_p$. We call elements of $S_k(\Gamma_0(p)\C_p)$ \begin{it}classical modular forms\end{it} and recall that space $S_{k}(\Gamma_0(p),\C_p)$ is contained in the space of overconvergent modular forms $M_k^{\dg}(r)$. Gouv\^ea conjectured in  \cite[Conjecture 3]{Gouvea_cont} that all small slopes of overconvergent modular forms come from forms in $S_{k}(\Gamma_0(p),\C_p)$, which was proven by Coleman \cite{Coleman_cl1}. We cite a special case of the result as the following proposition.
\begin{prop}\label{thm:smallslopes}
Suppose that $k\in\N$, let $p$ be a prime, and $r\in\C_p$. If $F\in M_k^{\dg}(r)$ has slope $y\leq k-1$, then $F\in S_{k}(\Gamma_0(p),\C_p)$. 
\end{prop}

The small slopes for weights $k$ and $\kappa\equiv k\pmod{(p-1)p^n}$ for $n$ sufficiently large coincide by another result of Coleman \cite[Theorem A, Theorem D]{Coleman_B}. A quantifying refinement of this statement was given by Wan \cite{Wan}, where a bound was given on how large $n$ needs to be. For $k=10$, the space $S_k$ is empty, and thus the space $S_k(\Gamma_0(p))$ consists of newforms, all eigenforms of the $U_p$-operator with eigenvalues $\pm p^{\frac{k-2}{2}} = p^4$ (see \cite[Theorem 4.6.17 (2)]{Miyake}). We thus expect that, for $\kappa = k \pmod{(p-1)p^n}$ with $n$ sufficiently large, the slope sequence starts with four. Let $p=5$ and pick $\kappa= 10+4\cdot5^2 = 110$.
Looking at the tables calculated by Gouv\^ea \cite{Gouvea_tables}, we find the slope sequence for $p=5$ at weight $110$ to be $4, 4, 4, 9, 10, 13, 14, 15$, and that indeed starts with four.
Alternatively, pick $p=3$, and $\kappa = 10+2\cdot3^5=496$. The slope sequence $p=3$ at weight $496$ starts with $4, 4, 9, 13, 13, 18, 23, 23, 27 \ldots$, and the leading (smallest) slope is again four as predicted.

Our  proof of Theorem \ref{thm_average_rate} consists of two observations. Firstly, we note that (a multiple of) $f$ belongs to $M_k^{\dg}(r)$ for any $r$. Then, we  observe that the slope sequence for $M_k^{\dg}(r)$ contains the slope $\frac{k-2}{2}$ and use Proposition \ref{thm:smallslopes} to conclude that all elements of $M_k^{\dg}(r)$ whose slopes are smaller than  $\frac{k-2}{2}$ must come from applying appropriate operators to elements of $S_k$. Namely, if $f\in S_{k}$ is a Hecke eigenform with eigenvalue $\lambda=\alpha+\beta$ with $\alpha$ and $\beta$ roots of the corresponding Hecke polynomial, then one naturally obtains eigenfunctions of $U_p$ by taking $f-\alpha f|V_p,\ f-\beta f|V_p\in S_{k}(\Gamma_0(p))$.
\begin{lem}\label{lem:Feigen}
Suppose that $f\in S_{k}$ is a Hecke eigenform and $\alpha$ and $\beta$ are the corresponding roots of the Hecke polynomial. Then 
\begin{equation*}
\left(f-\alpha f|V_p\right)|U_p=\beta\left(f-\alpha f|V_p\right),\quad
 \left(f-\beta f|V_p\right)|U_p=\alpha\left(f-\beta f|V_p\right).
\end{equation*}
\end{lem}
\begin{proof}
If the Hecke eigenvalue of $f$ under $T_p$ is $\lambda$, then we have
\[
\lambda f=f\big|T_p=f|U_p+p^{k-1}f|V_p.
\]
Thus
\begin{equation}\label{eqn:FjUp}
f|U_p=\lambda f - p^{k-1} f|V_p.
\end{equation}
 Then $\alpha\beta=p^{k-1}$, $\alpha+\beta=\lambda$, and, since $V_p\circ U_p$ is the identity, we conclude from \eqref{eqn:FjUp} that
\begin{align*}
\nonumber \left(f-\alpha f|V_p\right)|U_p&=\beta f-\alpha\beta f|V_p=\beta\left(f-\alpha f|V_p\right),\\
\nonumber \left(f-\beta f|V_p\right)|U_p&=\alpha f-\alpha\beta f|V_p=\alpha\left(f-\beta f|V_p\right).\qedhere
\end{align*}
\end{proof}
By Lemma \ref{lem:Feigen}, $\ord_p(\beta)$ and $\ord_p(\alpha)$ are slopes of elements of $S_{k}(\Gamma_0(p))$. We assume throughout that $\ord_p(\alpha)\leq \ord_p(\beta)$. Since $S_k$ has a basis of elements in $S_k(\SL_2(\Z),\Z)$, one can normalize a Hecke eigenform so that all of its Fourier coefficients are $p$-integral (see \cite[Theorem X.4.4]{Lang_mf}). The elements from Lemma \ref{lem:Feigen} then give a nice basis of eigenforms under $U_p$ for $S_{k}^{\operatorname{old}}(\Gamma_0(p),\C_p):=S_{k}(\C_p)\oplus S_{k}(\C_p)|V_p$.
\begin{lem}\label{lem:oldspacebasis}
There exists a field $K_p/\Q_p$ which is a finite Galois extension such that the space $S_k^{\operatorname{old}}(\Gamma_0(p),\C_p)$ has a basis of eigenfunctions under $U_p$ with eigenvalues and coefficients in $K_p$.
\end{lem}
\begin{proof}
Since $S_k(\SL_2(\Z),\C)$ has a basis in $S_k(\SL_2(\Z),\Z)$ and Hecke eigenforms have Fourier coefficients contained in a number field, we may construct a basis of Hecke eigenforms $f_1,\dots,f_d\in S_k(\SL_2(\Z),K)$ for some finite extension $K$ of $\Q$. We may extend $K$ to contain the roots $\alpha_j$ and $\beta_j$ of the corresponding Hecke polynomial for $f_j$. Embedding $K$ into $\C_p$ and taking the Galois closure, we obtain a field $K_p$ which is a finite Galois extension over $\Q_p$.

The space $S_k^{\operatorname{old}}(\Gamma_0(p),\C_p)$ has dimension $2d$, which can be seen by taking the basis $f_j$ and $f_j|V_p$ for $1\leq j\leq d$. We change to the basis $f_j-\alpha_j f_j|V_p$ and $f_j-\beta_j f_j|V_p$ for $1\leq j\leq d$. These all have coefficients in $K_p$ because $\alpha_j,\beta_j\in K_p$ and $f_j$ has coefficients in $K_p$. Moreover, the basis element $f_j-\alpha_j f_j|V_p$ (resp. $f_j-\beta_j f_j|V_p$) is an eigenfunction under $U_p$ with eigenvalue $\beta_j\in K_p$ (resp. $\alpha_j\in K_p$) by Lemma \ref{lem:Feigen}.
\end{proof}
Although $f$ is not itself an eigenfunction under $U_p$, repeated actions of $U_p$ on $f$ are highly congruent to repeated actions of $U_p$ on $f-\beta f|V_p$. 
\begin{lem}\label{lem:FUprepeat}
Suppose that $f\in S_{k}$ is a Hecke eigenform with Fourier coefficients in $\overline{\Q}$, normalized so that all of its Fourier coefficients are $p$-integral, and the roots $\alpha$ and $\beta$ of the corresponding Hecke polynomial are distinct and satisfy $\ord_p(\alpha)\leq \ord_p(\beta)$. Then we have
\[
f|U_p^m\equiv \frac{\alpha^{m+1}}{\alpha-\beta}\left(f-\beta f|V_p\right)\pmod{p^{\frac{m(k-1)}{2}-\varepsilon}}
\]
for some $\varepsilon\geq 0$ which is independent of $m$ and $\varepsilon=0$ if $\ord_p(\alpha)<\ord_p(\beta)$. 
\end{lem}
\begin{proof}
Since $\alpha\neq \beta$, we may write 
\begin{align*}
&f=\frac{1}{\alpha-\beta}\left(\alpha \left(f-\beta f|V_p\right) -\beta\left(f-\alpha f|V_p\right)\right).
\end{align*}
Applying $U_p^m$ to both sides and using the fact that $f-\beta f|V_p$ (resp. $f-\alpha f|V_p$) has eigenvalue $\alpha$ (resp. $\beta$) under $U_p$ by Lemma \ref{lem:Feigen} yields
\begin{equation*}
f|U_{p^m} =\beta^m f + \alpha\frac{\alpha^{m}-\beta^m}{\alpha-\beta}  \left(f-\beta f|V_p\right).
\end{equation*}
Since $\ord_p(\alpha)\leq \ord_p(\beta)$ and $\ord_p(\alpha)+\ord_p(\beta)=k-1$ (as $\alpha\beta=p^{k-1}$), we have $\ord_p(\beta)\geq \frac{k-1}{2}$, so (recalling that $f$ has $p$-integral Fourier coefficients)
\[
\beta^m f + \alpha\frac{\alpha^{m}-\beta^m}{\alpha-\beta}  \left(f-\beta f|V_p\right)\equiv  \alpha\frac{\alpha^{m}-\beta^m}{\alpha-\beta}  \left(f-\beta f|V_p\right) \pmod{p^{\frac{(k-1)m}{2}}}.
\]
Setting $\varepsilon:=\ord_p(\alpha-\beta)-\ord_p(\alpha)\geq 0$ (with $\varepsilon=0$ if $\ord_p(\beta)>\ord_p(\alpha)$) then gives the claim. 
\end{proof}
\section{Proof of Theorem \ref{thm_average_rate} and Corollary \ref{cor_main}}\label{sec:thmaverage}

As noted in the introduction, Corollary  \ref{cor_main} follows from Theorem \ref{thm_average_rate} immediately, so we only need to prove Theorem \ref{thm_average_rate}.

\begin{proof}[Proof of Theorem \ref{thm_average_rate}] 

Assume that the pole of $F$ at $\mathfrak{z}$ has order $D>0$.  We let $g$ be the element of $M_{p-1}$ satisfying the properties from  Lemma \ref{lem:supersingular}. Note that $Fg^D$ is a holomorphic cusp form of weight $k+D(p-1)$ defined over $\Z$ (because both $F$ and $g$ are defined over $\Z$). Writing $r^DF= \frac{r^D Fg^D}{g^D}$ we conclude that $r^D F \in M_k^{\dg}(r)$ (for any $r$ with $\ord_p(r)>0$) since it has a Katz expansion \eqref{eqn:KatzExpansion} (which consists of only one non-zero term with $b_D = Fg^D$).

We now pick $r$ such that \eqref{eq_growth_inequ} holds, and consider the slope sequence for $M_k^{\dg}(r). \text{ We have } S_k(\Gamma_0(p),\C_p) \subset M_k^{\dg}(r)$ and therefore the slope sequence contains all of the slopes coming from Hecke eigenforms in $S_k$ via Lemma \ref{lem:Feigen} and newforms of level $p$. Moreover, it follows from Proposition \ref{thm:smallslopes} that, since $\frac{k-2}{2}< k-1$, all elements in the slope sequence which are smaller than $\frac{k-2}{2}$ come from these classical modular forms. In other words, in the truncated asymptotic expansion \eqref{eq:tr_as} with $M=\frac{k-2}{2}-\delta$ (for some sufficiently small $\delta>0$ so that there are no slopes between $M$ and $\frac{k-2}{2}$) 
\begin{equation}\label{eqn:project}
e_{ <\frac{k-2}{2}} \left(r^D F\right) = \sum_{\substack{y\in\mathcal{Y}\\ \ord_p(y)<\frac{k-2}{2}}} e_y\left(r^DF\right),
\end{equation}
 Proposition \ref{thm:smallslopes} implies that all projections $e_y(r^D F)$ lie in $S_{k}(\Gamma_0(p),\C_p)$. The space $S_k^{\text{new}}(\Gamma_0(p))$ of weight $k$ newforms admits a basis which consists of Hecke eigenforms, and the slope of any weight $k$ Hecke newforms is $\frac{k-2}{2}$ by \cite[Theorem 4.6.17 (2)]{Miyake}. Hence the projection \eqref{eqn:project} only contains elements from the old space $S_{k}^{\operatorname{old}}(\Gamma_0(p),\C_p)$. By Lemma \ref{lem:oldspacebasis}, there exists a finite Galois extension $K_p/\Q_p$ such that \eqref{eqn:project} is spanned by a basis $f_1,\dots,f_d$ of eigenfunctions under $U_p$ with eigenvalues and coefficients in $K_p$. We may normalize the generators $f_1,\dots,f_d$ so that their coefficients are $p$-integral. Since $F$ has coefficients in $\Z\subset \Q_p\subseteq K_p$, Proposition \ref{prop:ProjectField} implies that the coefficients of $F_2:=e_{ <\frac{k-2}{2}} (r^D F)$ lie in $K_p$. Since $F_2\in S_{k}^{\operatorname{old}}(\Gamma_0(p),K_p)=S_{k}(K_p)\oplus S_{k}(K_p)|V_p$, we may decompose
\begin{equation}\label{eqn:decompose}
F_2=\sum_{j=1}^d \left(c_j\left(f_j-\alpha_jf_j|V_p\right)+d_j\left(f_j-\beta_jf_j|V_p\right)\right),
\end{equation}
with $c_j,d_j\in K_p$ and $\alpha_j,\beta_j$ the roots of the corresponding Hecke polynomial.

By the second statement of Proposition \ref{thm:TruncatedAsymptotic}, we have 
\begin{equation}\label{eqn:TruncateF-F2}
\left(F-F_2\right)|U_p^m\equiv 0\pmod{p^{\frac{(k-2)m}{2}-\varepsilon}}.
\end{equation}
We next show that there exist $g\in S_{k}(\SL_2(\Z),K_p)$ and $\varepsilon\geq 0$ such that 
\begin{equation}\label{eqn:f2FUprepeat}
F_2|U_p^m\equiv g|U_p^m\pmod{p^{\frac{(k-1)m}{2}-\varepsilon}}
\end{equation}
for all $m\in\N_0$ (and hence combining with the slightly weaker congruence in \eqref{eqn:TruncateF-F2} gives $(F-g)|U_{p}^m\equiv 0\Pmod{p^{\frac{(k-2)m}{2}-\varepsilon}}$). To construct $g$, we apply $U_p^{m}$ to both sides of \eqref{eqn:decompose}. If $\alpha_j=\beta_j$, then $\ord_p(\alpha)=\ord_p(\beta)=\frac{k-1}{2}$, so Lemma \ref{lem:Feigen} implies that, for some $\varepsilon\geq 0$,  
\[
\left(c_j\left(f_j-\alpha_jf_j|V_p\right)+d_j\left(f_j-\beta_jf_j|V_p\right)\right)|U_{p}^m\equiv 0\pmod{p^{\frac{(k-1)m}{2}-\varepsilon}}.
\]
For $\alpha_j\neq \beta_j$, we use the fact that $V_p\circ U_p$ is the identity to obtain 
\begin{multline*}
F_2|U_p^m \equiv \sum_{\substack{1\leq j\leq d\\ \alpha_j\neq \beta_j}} \Big(c_j\left(f_j|U_{p}^m-\alpha_jf_j|U_p^{m-1}\right)\\[-10pt]
+d_j\left(f_j|U_{p}^m-\beta_jf_j|U_p^{m-1}\right)\Big)\pmod{p^{\frac{(k-1)m}{2}-\varepsilon}}.
\end{multline*}
By Lemma \ref{lem:FUprepeat}, we conclude that  
\begin{align}
\nonumber F_2|U_p^m &\equiv \sum_{\substack{1\leq j\leq d\\ \alpha_j\neq \beta_j}} \Bigg(c_j\left(\frac{\alpha_j^{m+1}}{\alpha_j-\beta_j}\left(f_j-\beta_j f_j|V_p\right)-\frac{\alpha_j^{m+1}}{\alpha_j-\beta_j}\left( f_j-\beta_j f_j|V_p\right)\right)\\
\nonumber &\qquad \quad\qquad +d_j\left(\frac{\alpha_j^{m+1}}{\alpha_j-\beta_j} \left(f_j-\beta_j f_j|V_p\right)-\frac{\alpha_j^m\beta_j}{\alpha_j-\beta_j}\left(f_j-\beta_j f_j|V_p\right)\right)\Bigg)\\
\label{eqn:f2Uprepeat} &\equiv \sum_{\substack{1\leq j\leq d\\ \alpha_j\neq \beta_j}} d_j \alpha_j^{m}\left(f_j-\beta_j f_j|V_p\right)\pmod{p^{\frac{(k-1)(m-1)}{2}-\varepsilon}}.
\end{align}
Setting 
\[
h:=\sum_{j=1}^d d_j\frac{\alpha_j-\beta_j}{\alpha_j} f_j\in S_{k}(\SL_2(\Z),K_p),
\]
we note that the terms with $\alpha_j=\beta_j$ disappear and we may again use Lemma \ref{lem:FUprepeat} and then compare with \eqref{eqn:f2Uprepeat} to obtain 
\begin{equation*}
h|U_p^m =\sum_{\substack{1\leq j\leq d\\ \alpha_j\neq \beta_j}} d_j \frac{\alpha_j-\beta_j}{\alpha_j} f_j|U_p^m
\equiv  \sum_{\substack{1\leq j\leq d\\ \alpha_j\neq \beta_j}} d_j \alpha_j^{m} f_j\equiv F_2|U_p^m\pmod{p^{\frac{(k-1)(m-1)}{2}-\varepsilon}},
\end{equation*}
verifying \eqref{eqn:f2FUprepeat} with $\varepsilon\mapsto \varepsilon+\max\{\frac{k-1}{2}-\frac{m}{2},0\}$. 

We therefore obtain 
\[
\left(F-h\right)|U_p^m\equiv 0\pmod{p^{\frac{(k-1)(m-1)}{2}-\varepsilon}}.
\]
We next take the Galois trace from $K_p$ down to $\Q_p$ on the above congruence Since $F$ has coefficients in $\Z$, Galois conjugation does not affect $F$. Letting $f$ be $\frac{1}{\left[K_p:\Q_p\right]}$ times the Galois trace of $h$, we therefore obtain 
\[
\left(\left[K_p:\Q_p\right] F-\left[K_p:\Q_p\right] f\right)|U_p^m\equiv 0\pmod{p^{\frac{(k-1)(m-1)}{2}-\varepsilon}}.
\]
Dividing both sides by $[K_p:\Q_p]$ and letting $\varepsilon\mapsto \varepsilon + \ord_p([K_p:\Q_p])$, we obtain 
\[
\left(F-f\right)|U_p^m\equiv 0\pmod{p^{\frac{(k-1)(m-1)}{2}-\varepsilon}}. 
\]
Since $f$ is $\frac{1}{\left[K_p:\Q_p\right]}$ times the Galois trace of $h$ from $K_p$ down to $\Q_p$, its Fourier coefficients lie in $\Q_p$. Finally, since $S_k$ has a basis of forms from $S_k(\SL_2(\Z),\Z)$, there exists $f_m\in S_k(\SL_2(\Z),\Q)$ with $f\equiv f_m(\!\!\!\mod{p^{\frac{m(k-2)}{2}-\varepsilon}})$, giving the second claim.\qedhere
\end{proof}

\end{document}